\documentclass[11pt]{article}

\usepackage[margin=1.08in]{geometry}
\usepackage{amsmath,amssymb,amsthm}
\usepackage{booktabs}
\usepackage{float}
\usepackage{graphicx}
\usepackage[expansion=false]{microtype}
\usepackage{xcolor}
\usepackage{hyperref}
\usepackage[ruled,vlined,linesnumbered]{algorithm2e}
\usepackage[nameinlink,capitalize,noabbrev]{cleveref}
\usepackage{tikz}
\usetikzlibrary{matrix}
\usepackage{etoolbox}
\usepackage{listings}

\hypersetup{
  colorlinks=true,
  linkcolor=blue!55!black,
  citecolor=blue!55!black,
  urlcolor=blue!60!black,
  pdftitle={Local dimension testing via quadratic obstructions},
  pdfauthor={Taylor Brysiewicz and Rainer Sinn}
}

\newtheorem{theorem}{Theorem}[section]
\newtheorem{corollary}[theorem]{Corollary}

\theoremstyle{definition}
\newtheorem{example}[theorem]{Example}

\AtBeginEnvironment{example}{%
  \pushQED{\qed}%
}
\AtEndEnvironment{example}{\popQED}

\theoremstyle{remark}
\newtheorem{remark}[theorem]{Remark}

\crefname{algocf}{Algorithm}{Algorithms}
\Crefname{algocf}{Algorithm}{Algorithms}

\newcommand{\CC}{\mathbb{C}}
\newcommand{\QQ}{\mathbb{Q}}
\newcommand{\coker}{\operatorname{coker}}
\newcommand{\rank}{\operatorname{rank}}
\newcommand{\Span}{\operatorname{span}}
\newcommand{\p}{\mathbf{p}}
\newcommand{\x}{\mathbf{x}}
\newcommand{\w}{\mathbf{w}}
\newcommand{\y}{\mathbf{y}}
\newcommand{\zero}{\mathbf{0}}
\renewcommand{\u}{\mathbf{u}}
\renewcommand{\v}{\mathbf{v}}
\newcommand{\mydef}[1]{{\color{blue}#1}}

\definecolor{maroon}{RGB}{133,5,63}
\definecolor{forestgreen}{RGB}{34,139,34}

\lstdefinelanguage{M2}{
  basicstyle=\small\ttfamily,
  alsoletter=",
  classoffset=1,
  comment=[l]{\--},
  commentstyle=\color{gray},
  keywords={
    solve, differentiate, subs, sub, real_solutions, det, sum, max,
    count, conditional_count, map, filter, map_filter, union, zip,
    product, is_finite, is_successful, is_nonsingular, track, is_real,
    is_success, first, evaluate, flatten, bitmask_filter, entries,
    unique, join, accumulate, polynomial_interpolants, reverse,
    coefficients, vcat, solution_from_necklace, length, prod, compress,
    total_degree_start_solutions, degrees, variables, iterate, struct,
    collect, convert, stretched_cube, stretched_cubes, weight_vector,
    weight_vectors, perm_to_segments, perm_to_mixedcell,
    mixed_cell_iterator, rand_approx_unit, norm, fixed,
    polyhedral_system, reduce, findall, eachrow, matrix, peek,
    coefficientRing, numgens, ring, print, ideal, toString
  },
  keywordstyle=\color{cyan},
  classoffset=2,
  morekeywords={
    @var, time, for, end, if, in, from, to, while, else, begin,
    list, new, do, scan, apply, close, openOut
  },
  keywordstyle=\color{maroon},
  classoffset=3,
  morekeywords={using, function, return, const, needsPackage},
  keywordstyle=\color{blue},
  classoffset=4,
  morekeywords={julia, >},
  keywordstyle=\color{forestgreen},
  classoffset=5,
  morekeywords={polySystem, point, newton, certifyRealSolution},
  keywordstyle=\color{purple},
  xleftmargin=0.2cm,
  xrightmargin=1em,
  columns=fullflexible,
  keepspaces=true
}

\title{\textbf{Local dimension testing via quadratic obstructions}}
\author{Taylor Brysiewicz \and Rainer Sinn}
\date{}

\begin{document}

\maketitle

\begin{abstract}
We give an algorithm, based on second-order necessary conditions, to test whether a point is isolated on an algebraic set. Using this subroutine, we develop an algorithm that bounds the local dimension of an algebraic set at a point. The principal computation in our isolation test is governed by the nullity of the Jacobian, avoiding the combinatorial growth of classical methods. For exact input, the bound returned by the algorithm is certified. We compare our approach with classical methods. As an application, we show that the realization space of the $24$-cell has the expected local dimension at its symmetric regular realization, completing the local dimension program for regular $4$-polytopes initiated by Rastanawi, Sinn, and Ziegler. 
\end{abstract}

\section{Introduction}
The \textit{local dimension} of an algebraic set $X \subseteq \mathbb{C}^n$ at a point $\p \in X$ is the largest dimension of a component of $X$ containing $\p$. In rigidity theory, local dimension measures local degrees of freedom \cite{AdrianHimmelmannWinterZhang2026,ConnellyWhiteley1996}. Local dimension also appears as a core concept in early algorithms for \textit{numerical irreducible decomposition} \cite{SommeseVerscheldeWampler2001}, a foundational task in numerical algebraic geometry~\cite{HauensteinSommese2017}. Namely, it is used to remove extraneous points from witness supersets in the cascade algorithm~\cite{BatesHauensteinPetersonSommese2009,SommeseVerscheldeWampler2001}.

Algorithms for computing local dimensions, called \textit{local dimension tests}, extract geometric data that may not be visible from first-order information. Let $F$ be a set of $m$ defining equations for $X$, and let $J$ be the Jacobian of $F$ evaluated at $\p$. The local dimension of $X$ at $\p$ is bounded above by the nullity $n - \rank(J)$ of $J$, the dimension of the Zariski tangent space of the scheme defined by $F$ at $\p$ \cite[Chapter 9.6, Theorem 8]{zbMATH06423120}. A lower bound of $n-m$ follows from the observation that imposing one equation decreases the dimension  by at most one. In summary,
\[
n-m \leq \dim_{\p}(X) \leq n-\rank(J).
\]
 An \textit{isolation test}, on the other hand, determines only whether the local dimension is zero. Thus, every local dimension test specializes to an isolation test. Conversely, an isolation test yields a local dimension test by slicing $X$ with generic linear spaces through $\p$: the local dimension is the smallest codimension of a slice that makes $\p$ isolated.

A one-sided isolation test  returns either \textit{isolated} or \textit{inconclusive}. Since genericity of linear slices can rarely be certified in practice, a result of \textit{isolated} applied to a codimension-$k$ slice of $X$ through $\p$ establishes only an upper bound of $k$ on the local dimension $\dim_{\p}(X)$.

We give such a one-sided isolation test in \Cref{alg:quadratic-isolation}. The foundational observation underlying the algorithm is that a nonisolated point must admit a curve through it and that a local parametrization of such a curve must vanish identically when evaluated at the defining equations of $X$. Extracting the lowest-order term of that evaluation shows that the sliced Jacobian must have a nontrivial kernel, i.e., that the sliced Zariski tangent space is nontrivial. Extracting the second-lowest-order term produces quadratic equations in the coordinates of the kernel. If these equations admit no nonzero solution, then no such curve exists and $\p$ is isolated.

The isolation test of \Cref{alg:quadratic-isolation} induces the local dimension test of \Cref{alg:localdimension}, which produces bounds on the local dimension $\dim_{\p}(X)$. Either algorithm may be executed on symbolic or numerical input. If the input and steps are exact, then the output of the algorithm is certified.

\begin{theorem}[Local dimension bounds]\label{thm:localdimensionbounds} Let $F \subseteq \mathbb{Q}[x_1,\ldots,x_n]$ be a set of $m$ polynomials cutting out a variety $X\subseteq \mathbb{C}^n$ containing $\p \in \mathbb{Q}^n$. 
\Cref{alg:localdimension} terminates on input $(F,\p)$ and returns an output $[\lambda,k]$ such that
\[
  \lambda\leq\dim_{\p}(X)\leq k,
  \qquad
  \lambda=\max\{n-m,0\}.
\]
\end{theorem}

We implemented \Cref{alg:localdimension} in \texttt{Macaulay2} \cite{Macaulay2} for one example of interest: the symmetric regular  realization of the $24$-cell in its realization space. Rastanawi, Sinn, and Ziegler showed that it is singular and bounded its local dimension between $48$ and $50$~\cite[Theorem~5.10 and Remark~5.11]{RastanawiSinnZiegler2021}. We apply \Cref{alg:localdimension} to prove that its local dimension is $48$.
\begin{theorem}
\label{thm:24cell}
The symmetric regular $24$-cell has local dimension $48$ in its realization space.
\end{theorem}

\begin{figure}[!htpb]
\centering
\includegraphics[scale=0.3]{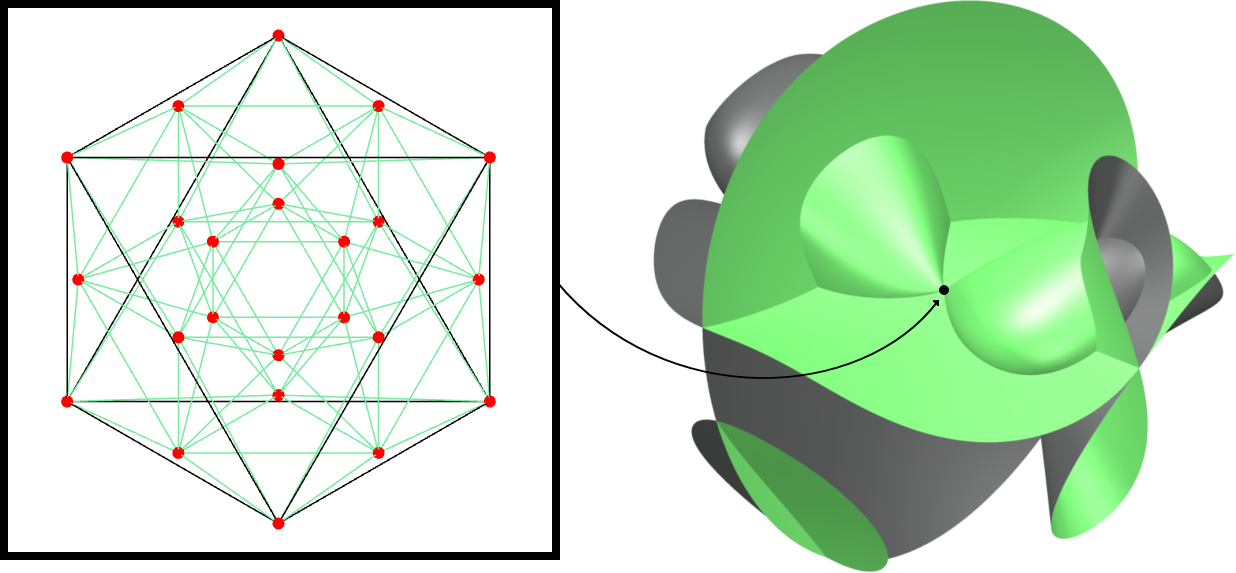}
\caption{(Left) A $3$D projection of the edge skeleton of the symmetric
regular $24$-cell (created in \texttt{polymake}
\cite{GawrilowJoswig2000}).  (Right) A cartoon of a two-dimensional slice of
the realization space of all $24$-cells, with the symmetric regular realization depicted as a singular point (created in \texttt{surf}
\cite{Surf2015}).}
\end{figure}
\noindent This computation completes the assessment of the local dimensions of the six regular $4$-polytopes, mostly carried out in \cite{RastanawiSinnZiegler2021}. The code in the appendix proves \autoref{thm:24cell}.

\noindent \textbf{History and relationship to previous literature:}
In $2021$, we attempted to prove \Cref{thm:24cell} by implementing the isolation test of \cite{BatesHauensteinPetersonSommese2009}. That work uses multiplicity matrices to establish the multiplicity structure of an isolated point and thereby determine whether it is isolated. We implemented the ideas of \cite{HaoSommeseZeng2013}, which substantially improve that framework by reining in the combinatorial growth of the multiplicity matrices involved. Our computation improved the possible local dimension range to two possibilities: $48$ or $49$. However, our software was not able to establish $48$ after weeks of running on a supercomputer. The quadratic obstruction calculation proves \Cref{thm:24cell} almost instantaneously on a personal laptop.

The reason \Cref{alg:localdimension} is so fast compared to our previous implementation is that the quadratic isolation test works in the coordinates of $\ker(J)$ and is only one-sided. The classical local dimension test grows combinatorially with the number of degree-$d$ monomials in $n$ variables. If the point is isolated, the classical algorithm is never inconclusive but may need to reach a large value of $d$ before drawing a conclusion. Our algorithm, however, has a cost governed primarily by the nullity of $J$ and is often inconclusive. In the case of the $24$-cell, this nullity is $4$, and the algorithm is conclusive.

The quadratic obstructions underlying our algorithms already appear in rigidity theory. Connelly and Whiteley
developed a second-order stress test for tensegrity frameworks
\cite{ConnellyWhiteley1996} and in very recent work, Adrian-Himmelmann, Winter, and
Zhang adapt this test to polytopes with edge-length and coplanarity constraints
\cite{AdrianHimmelmannWinterZhang2026}. Neither work, however, frames the quadratic obstruction equations as a general isolation test. 

\noindent\textbf{AI statement:}
In July 2026, the first author provided their 2021 implementation of a local dimension test to OpenAI's GPT-5.6 Sol model and asked for improvements that would make the computation tractable. The model suggested the quadratic obstruction approach that led to \Cref{alg:quadratic-isolation} and identified it as the quadratic term of the finite-dimensional Lyapunov--Schmidt reduction; see, for example, \cite[Chapter~VII]{GolubitskySchaeffer1985}. The model was then used for mathematical discussion, code development, and editorial assistance. Both authors  verified all mathematical claims and computations and take full responsibility for the content of this manuscript.

\noindent \textbf{Organization of the paper:} The paper is organized as follows. In \cref{sec:quadratic-test}, we formulate
the background for the quadratic isolation test.
\Cref{sec:algorithms} formulates the isolation and local dimension tests as algorithms,
\cref{sec:examples} compares them with classical local dimension calculations, and
\cref{sec:24cell} gives the exact computation for the symmetric regular $24$-cell. \Cref{app:24cell-computation} contains exact \texttt{Macaulay2} \cite{Macaulay2} code for verifying \Cref{thm:24cell}.

\section*{Acknowledgements}
TB  was supported by NSERC Discovery Grant RGPIN-
2023-03551.

\section{Setup, slicing, and quadratic isolation}\label{sec:quadratic-test}

Fix a set of $m$ polynomials in $n$ variables $\x=(x_1,\ldots,x_n)$:
\begin{equation*}
  F=\{f_1,\ldots,f_m\}\subseteq
  \CC[x_1,\ldots,x_n]=\CC[\x].
\end{equation*}
We denote the affine variety in $\CC^n$ cut out by $F$ by
\[
  X=\mathcal V(F)
  =\{\mathbf{a}\in\CC^n\mid F(\mathbf{a})=\zero\} \subseteq \CC^n.
\]
The variety $X$ decomposes into irreducible components. Given $\p \in X$, the largest dimension of a component containing $\p$ is the \mydef{local dimension of $X$ at $\p$}, denoted $\dim_{\p}(X)$. We translate $\p$ to the origin when convenient.

A point
$\p\in X$ is \mydef{isolated} if some neighborhood of $\p$ contains no other
point of $X$. The local dimension of $X$ may be characterized entirely in these terms:
if $\ell_1,\ldots,\ell_n$ are generic affine linear forms vanishing at $\p$, the \mydef{local dimension} of $X$ at $\p$ is
\begin{equation}\label{eq:slicing-characterization}
  \dim_{\p}(X)
  =
  \min\left\{
    k\mid \p\text{ is isolated in }
    X\cap\mathcal V(\ell_1,\ldots,\ell_k)
  \right\}.
\end{equation}

\subsection{Useful bounds}
We now establish several useful bounds on $\dim_{\p}(X)$. First, observe that since
$X$ is nonempty and cut out by $m$ equations $F$ in $n$ variables, \cite[Chapter I, \S6.2]{Shafarevich2013} implies that
\begin{equation*}
 \lambda =\max\{n-m,0\}\leq  \dim_{\p}(X) \leq \dim(X).
\end{equation*}
First-order information  bounds the local dimension of $X$ at $\p$. Let 
$J = ( \partial_j f_i (\p))_{i,j}$ be the Jacobian of $F$ evaluated at $\p$. The dimension $n-\rank(J)$ of the Zariski tangent space of the scheme defined by $F$ at $\p$ bounds the local dimension  \cite[Chapter 9.6, Theorem 8]{zbMATH06423120}:
\begin{equation*}
\dim_{\p}(X) \leq n-\rank(J)
\end{equation*}
Intersecting with a hypersurface $\mathcal H$ containing $\p$ decreases local dimension by at most one:
\[
  \dim_{\p}(X)
  \leq
  \dim_{\p}(X\cap\mathcal H)+1.
\]
Consequently, if $\ell_1,\ldots,\ell_k$ are linear forms vanishing at $\p$,
then
\begin{equation}
\label{eq:slice-upper}
  \dim_{\p}(X)
  \leq
  \dim_{\p}\!\left(X\cap\mathcal V(\ell_1,\ldots,\ell_k)\right)+k,
\end{equation}
and if $\p$ is isolated in $X \cap \mathcal V(\ell_1,\ldots,\ell_k)$, then $
\dim_{\p}(X)\leq k$. Since \eqref{eq:slice-upper} does not depend on the genericity of $\ell_1,\ldots,\ell_k$, a special slice may give a nonsharp upper bound but never a false one.

\subsection{Quadratic obstructions}
The key observation of this work is that the local dimension of $X$ at $\p$ is positive if and only if there is a curve in $X$ containing $\p$. 
Interpret $F$ as the polynomial map
\[
  F:\CC^n\longrightarrow\CC^m,\qquad
  \x\longmapsto
  \bigl(f_1(\x),\ldots,f_m(\x)\bigr),
\]
and write the Taylor expansion of the map $F$ at $\p$ as
\[
  F(\p+\x)
  =F(\p)+F_1(\x)+F_2(\x)+F_3(\x)+\cdots+F_d(\x),
\]
where $F_k$ is a vector of homogeneous polynomials of degree $k$ and $d=\textrm{max}(\textrm{deg}(f_i))$.

If $\p$ lies on a positive-dimensional component of $X$, then there is an algebraic curve through $\p$. Such a curve admits a local parametrization
\begin{equation}
\label{eq:curve}
  \gamma(t)=\p+t^r\u+\sum_{j>r}t^j\v_j,
  \qquad \u\neq\zero.
\end{equation}
In particular, $F(\gamma(t))\equiv\zero$.  The coefficients of $t^r$ and $t^{2r}$ in 
$F(\gamma(t))$ are 
\[
  \textrm{coefficient of }t^r: J\u,\quad \quad \quad \quad 
  \textrm{coefficient of }t^{2r}: J\v_{2r}+F_2(\u).
\]
The existence of $\gamma(t)$ requires both coefficients to vanish. The vanishing of the first coefficient recovers the first-order condition that $\u\in \ker(J)$. The vanishing of the second coefficient gives the second-order condition that $F_2(\u)\in\operatorname{im}(J)$. Similarly, there are higher-order conditions. Our algorithm focuses on the quadratic term, which 
we encode in terms of the map
\[
  q:\textrm{ker}(J)\longrightarrow\coker(J),\qquad
  q(\u)=\pi\bigl(F_2(\u)\bigr),
\]
where $\pi$ is a quotient map from $\CC^m$ to $\coker(J)=\mathbb{C}^m/\textrm{im}(J)$. 
 If
$q(\u)\neq\zero$, then $\u$ cannot occur as the leading
direction of a curve \eqref{eq:curve} on $X$.    Consequently,
if $q^{-1}(\zero)=\{\zero\}$, then $\p$ is isolated in $X$. A \mydef{quadratic test} verifies this equality and is made precise in \Cref{alg:quadratic-isolation}. The converse is not true. If $q(\u) = \zero$, then higher-order terms may still prevent such a curve from existing.

\begin{theorem}[Quadratic isolation test]\label{thm:quadratic-certificate}
Let $F:\CC^n\to\CC^m$ be a polynomial map, let
$\p\in\mathcal V(F)$, set
$J=\operatorname{Jac}(F)(\p)$, and let $F_2$ be the
quadratic part of $F(\p+\x)$. 
Suppose that the equations
\begin{equation}\label{eq:quadratic-condition}
  J\u=\zero,
  \qquad
  J\v+F_2(\u)=\zero
\end{equation}
have no solution $(\u,\v)\in\CC^n\times\CC^n$ with $\u\neq\zero$.  Then
$\p$ is isolated in $\mathcal V(F)$.
\end{theorem}

\begin{proof}
Suppose instead that $\p$ is not isolated.  Then an algebraic curve on
$\mathcal V(F)$ passes through $\p$. By passing to the normalization of the curve, we obtain a nonconstant local parametrization
\[
  \gamma(t)=\p+t^{r}\u+\sum_{j>r}t^{j}\v_j,
  \qquad \u\neq\zero.
\]
The coefficient of $t^r$ in $F(\gamma(t))$ is $J\u$, so $J\u=\zero$.
The coefficient of $t^{2r}$ is
\[
  J\v_{2r}+F_2(\u),
\]
where we take $\v_{2r}=\zero$ if that term is absent.  Since
$F(\gamma(t))\equiv\zero$, the pair $(\u,\v_{2r})$ is a solution of
\eqref{eq:quadratic-condition} with $\u\neq\zero$, a contradiction.
\end{proof}
The equations in \eqref{eq:quadratic-condition} are easily written in
coordinates on the kernel of the Jacobian. Let the columns of
$B\in\operatorname{Mat}_{n\times b}(\CC)$ form a basis of $\ker(J)$, and let
the columns of $L\in\operatorname{Mat}_{m\times c}(\CC)$ form a basis of
$\ker(J^T)$. Let $\w=(w_1,\ldots,w_b)$ be coordinates on $\ker(J)$, so that
$\u=B\w$. Then
\begin{equation}
\label{eq:derivingq}
  F_2(B\w)\in\operatorname{im}(J)
  \quad\Longleftrightarrow\quad
  q(\w)=L^TF_2(B\w)=\zero.
\end{equation}
Since $F_2: \mathbb{C}^n \to \mathbb{C}^m$ is a homogeneous quadratic map, the equation 
\begin{equation}
\label{eq:obstruction}
q(\w)=(q_1(\w),\ldots,q_c(\w))=\zero
\end{equation}
involves $c$ homogeneous quadratics in $b$ variables. We call \eqref{eq:obstruction} the \mydef{quadratic obstruction equations}.
The condition in \cref{thm:quadratic-certificate} is simply
\begin{equation*}
  \mathcal V(q_1,\ldots,q_c)=\{\zero\}.
\end{equation*}
Using slices, \cref{thm:quadratic-certificate} immediately gives the \mydef{quadratic obstruction 
local dimension test}.

\begin{corollary}[Quadratic obstruction local dimension bounds]\label{cor:quadratic-dimension}
Fix $F \subseteq \mathbb{C}[\x]$, $\p = \textbf{0}\in \mathcal V(F)$, and take $\ell_1,\ldots,\ell_k$ to be any homogeneous linear forms. The quadratic obstruction equations of \[F^{(k)} = F \cup \{\ell_i\}_{i=1}^k\]
consist of $c=|F|+k-\textrm{rank}(\textrm{Jac}(F^{(k)})(\p))$ equations $q_1,\ldots,q_c$ in $b = \dim(\ker(\textrm{Jac}(F^{(k)})(\p)))$ variables $w_1,\ldots,w_b$. If
$\mathcal V(q_1,\ldots,q_c) = \{\zero\}$, then
\[
  \dim_{\p}\mathcal V(F)\leq k.
\]
\end{corollary}

\begin{proof}
\Cref{thm:quadratic-certificate} says that $\p$ is isolated in the slice by \eqref{eq:derivingq}. Applying \eqref{eq:slice-upper} proves the result.
\end{proof}

\begin{remark}
After constructing the quadratic obstruction equations
$q_1,\ldots,q_c$, deciding whether
$\mathcal V(q_1,\ldots,q_c)=\{\zero\}$ is a standard task in computer
algebra.  Since the equations are homogeneous, their zero set is a cone
containing $\zero$.  It therefore suffices to check whether
$\dim\bigl(\langle q_1,\ldots,q_c\rangle\bigr)=0$.
This can be done,
for example, in \texttt{Macaulay2} \cite{Macaulay2}.
\end{remark}

\section{Algorithms and correctness}\label{sec:algorithms}
We now record the quadratic isolation calculation as
\Cref{alg:quadratic-isolation}.  \Cref{alg:localdimension} applies it as a
subroutine to determine certified upper bounds on the local dimension.  Both
are one-sided: failure to find a quadratic certificate may mean only that a
higher-order test is needed.

\begin{algorithm}[H]
\caption{\textsc{QuadraticIsolation}$(F,\p)$}
\label{alg:quadratic-isolation}

\SetKwInOut{Input}{Input}
\SetKwInOut{Output}{Output}

\Input{A polynomial map $F:\CC^n\to\CC^m$\\
       A point $\p\in\mathcal V(F)$}
\Output{\textsc{isolated}  or
        \textsc{inconclusive}}

Replace $F(\x)$ by $F(\p+\x)$, so that $\p$ becomes the origin, and set $J=\operatorname{Jac}(F)(\zero)$\;

Compute $B$ and $L$ whose columns are bases of
$\ker(J)$ and $\ker(J^T)$, respectively, and set $b=\dim\ker(J)$ and $c=\dim\ker(J^T)$\;

\textbf{if} $b=0$ \textbf{then} \textbf{return} \textsc{isolated}\;

Let $F_2$ be the homogeneous quadratic part of $F$\;
Construct the quadratic obstruction equations
$q(\w)=L^TF_2(B\w)$ in $\w=(w_1,\ldots,w_b)$\;

Set
$I_q=\langle q_1,\ldots,q_c\rangle
\subseteq\CC[w_1,\ldots,w_b]$\;

\textbf{if} $\dim(I_q)=0$ \textbf{then} \textbf{return} \textsc{isolated}\;

\Return{\textsc{inconclusive}}\;
\end{algorithm}

\begin{theorem}[Termination and one-sided correctness]
\label{thm:algorithm-correctness}
\Cref{alg:quadratic-isolation} terminates and its output is correct.  
If $b=0$, it returns \textup{\textsc{isolated}}.  If $b>0$,
let $q(\w)=L^TF_2(B\w)$ be the quadratic obstruction equations constructed
by the algorithm.  It returns \textup{\textsc{isolated}} if and only if
\[
  \mathcal V(q_1,\ldots,q_c)=\{\zero\}.
\]
Whenever it returns \textup{\textsc{isolated}}, the point $\p$ is isolated in
$\mathcal V(F)$.  \textup{\textsc{Inconclusive}} makes no assertion about
whether $\p$ is isolated.
\end{theorem}

\begin{proof}
If $b=0$, then $J$ has full rank, so  $\p$ is isolated.  Suppose now that $b>0$.
Computing the dimension of an ideal terminates via a Gr\"obner basis calculation.
Since the quadratic obstruction equations are homogeneous, their common zero
set is a cone containing the origin.  Hence $\dim(I_q)=0$ exactly when they
have no nonzero common solution.  \Cref{thm:quadratic-certificate} then
proves the isolation claim. The system $F(x)=x^3$ is an example that shows why the remaining
output must be inconclusive.
\end{proof}

The induced one-sided local dimension test uses \Cref{alg:quadratic-isolation} as a subroutine.

\begin{algorithm}[H]
\caption{\textsc{QuadraticLocalDimensionBounds}$(F,\p)$}
\label{alg:localdimension}

\SetKwInOut{Input}{Input}
\SetKwInOut{Output}{Output}

\Input{A polynomial map $F:\CC^n\to\CC^m$\\
       A point $\p\in\mathcal V(F)$}
\Output{An interval containing $\dim_{\p}\mathcal V(F)$}

Set $\lambda=\max\{n-m,0\}$\;

Choose linearly independent affine-linear polynomials
       $\ell_1,\ldots,\ell_n$ vanishing at $\p$\;

\For{$k=\lambda,\lambda+1,\ldots,n$}{
  Set $F^{(k)}=(F,\ell_1,\ldots,\ell_k)$\;
  Run \textsc{QuadraticIsolation}$(F^{(k)},\p)$\;
  \If{the output is \textsc{isolated}}{
    \Return{$[\lambda,k]$}\;
  }
}
\end{algorithm}

\begin{proof}[Proof of Theorem~\ref{thm:localdimensionbounds}]
Every irreducible component of $\mathcal V(F)$ through $\p$ has codimension
at most $m$, so the local dimension is at least
$\lambda=\max\{n-m,0\}$.  At $k=n$, the added affine-linear equations have
Jacobian rank $n$, so \cref{alg:quadratic-isolation} returns
\textsc{isolated}; hence the algorithm terminates.  If it stops at some smaller $k$,
then the variety sliced with $\mathcal V(\ell_1,\ldots,\ell_k)$ is isolated at $\p$ and equation
\eqref{eq:slice-upper} gives $\dim_{\p}\mathcal V(F)\leq k$.
\end{proof}

\begin{remark}[One-sidedness]
For a general choice of $\ell_1,\ldots,\ell_n$, the first geometrically
isolated slice occurs at $k=\dim_{\p}\mathcal V(F)$ by
\eqref{eq:slicing-characterization}.  Its quadratic obstruction equations may
nevertheless have nonzero solutions.  For example, $F(x)=x^3$ has an isolated
zero at the origin but no quadratic terms.  Therefore the output $k$ may be
strictly larger than the local dimension. One may complete this test to a true local dimension test by (1) assuming genericity of the slices and (2) either implementing higher-order versions of \Cref{alg:quadratic-isolation} or using a classical local dimension test as a fallback when \Cref{alg:localdimension} does not identify the local dimension. 
\end{remark}

\begin{remark}[Regarding exactness]
\Cref{thm:localdimensionbounds} assumes that the input polynomial system and point are both given exactly, i.e., over the rationals. This hypothesis ensures that  \Cref{alg:quadratic-isolation} and \Cref{alg:localdimension}  can be implemented correctly. Over $\mathbb{R}$ and $\mathbb{C}$, the algorithms are still correct as abstract algorithms, however, the numerical subroutines for finding a basis for the kernel of $J$, computing the obstruction equations, or establishing that $\dim(I_q) = 0$ may not be certifiable. 
\end{remark}

\section{Comparison with the classical local dimension test}
\label{sec:examples}

The numerical local dimension test of Bates, Hauenstein, Peterson, and
Sommese \cite{BatesHauensteinPetersonSommese2009} uses the multiplicity method of \cite{DaytonZeng2005}.  We first describe this method, then compare the classical test with the quadratic test on several
examples.

\subsection{The classical multiplicity matrix}
For $d\geq 1$, consider the matrix $M_d$ whose columns are indexed by normalized
differential functionals
$
  \frac{1}{\alpha!}\partial^\alpha$ for $|\alpha|\leq d,
$
and whose rows are indexed by 
$
  \x^\beta f_i,
 $ for $i=1,\ldots,m$ and $|\beta|<d.
$  Following  \cite{BatesHauensteinPetersonSommese2009},
we call $M_d$ the $d$-th order \mydef{multiplicity matrix}. The $((i,\beta),\alpha)$-th entry is
\[(M_d)_{(i,\beta),\alpha} =
  \frac{1}{\alpha!}\partial^\alpha
  \bigl(\x^\beta f_i\bigr)
  (\zero).
\]
\begin{example}
Take $F=(f_1,f_2)=(x^2,y^3)$ and $\p$ the origin.  The monomials indexing columns
below denote the corresponding normalized differentials.  The third-order multiplicity matrix is
\[M_3 = {\footnotesize{
\begin{array}{r|rrrrrrrrrr}
 &1&x&y&x^2&xy&y^2&x^3&x^2y&xy^2&y^3\\ \hline
f_1       &0&0&0&1&0&0&0&0&0&0\\
xf_1      &0&0&0&0&0&0&1&0&0&0\\
yf_1      &0&0&0&0&0&0&0&1&0&0\\
x^2f_1    &0&0&0&0&0&0&0&0&0&0\\
xyf_1     &0&0&0&0&0&0&0&0&0&0\\
y^2f_1    &0&0&0&0&0&0&0&0&0&0\\ \hline
f_2       &0&0&0&0&0&0&0&0&0&1\\
xf_2      &0&0&0&0&0&0&0&0&0&0\\
yf_2      &0&0&0&0&0&0&0&0&0&0\\
x^2f_2    &0&0&0&0&0&0&0&0&0&0\\
xyf_2     &0&0&0&0&0&0&0&0&0&0\\
y^2f_2    &0&0&0&0&0&0&0&0&0&0
\end{array}.}}
\]
\end{example}
A vector $\eta=(\eta_\alpha)_{|\alpha|\leq d}$ determines the differential
functional
\[
  \Lambda_\eta
  =
  \sum_{|\alpha|\leq d}
  \eta_\alpha\frac{1}{\alpha!}\partial^\alpha\bigg|_{\zero}.
\]

The vector $M_d\cdot\eta$ records the values of $\Lambda_\eta$ on the row
polynomials: its $(i,\beta)$-th entry is
$\Lambda_\eta(\x^\beta f_i)$.  Hence $M_d\cdot\eta=\zero$ precisely when
$\Lambda_\eta$ annihilates every row of $M_d$.  Since $\Lambda_\eta$ has
order at most $d$, the monomial multiples in $M_d$ ensure that
$\Lambda_\eta$ annihilates the ideal generated by $F$.

The space of functionals satisfying $M_d\cdot\eta=\zero$ 
has dimension
$\mu_d=\dim\ker(M_d)$.
If $\p$ is isolated, the sequence $\mu_1\leq\mu_2\leq\cdots$ eventually
stabilizes at $\mu_k$, the multiplicity of $\p$. We call $k$ the \mydef{stabilization depth} and $k+1$ the \mydef{detection depth}. If $\p$ is not isolated, the
sequence continues to grow. More precisely, \cite[Theorem A.1]{BatesHauensteinPetersonSommese2009}
shows that two consecutive multiplicities are equal if and only
if $\p$ is isolated; in that case their common value is the multiplicity
of $\p$. The classical isolation test computes
successive multiplicity matrices until either two consecutive nullities
agree, proving isolation, or a known upper bound for the multiplicity of an
isolated point is exceeded, proving nonisolation.
\begin{example}
Continue with our example $F=(f_1,f_2)=(x^2,y^3)$ at the origin.  Ordering
the rows and columns by degree gives the nested sequence
\[
  M_1\hookrightarrow M_2\hookrightarrow M_3\hookrightarrow M_4.
\]
Below is $M_4$ with its zero rows suppressed.  The boxes indicate the
 submatrices $M_1,M_2$, and $M_3$.
\[
\begin{tikzpicture}[baseline=(M.center)]
\matrix (M) [
  matrix of math nodes,
  nodes={
    font=\footnotesize,
    minimum width=1.5em,
    minimum height=1em,
    inner sep=0pt
  },
  column sep=0pt,
  row sep=0pt
] {
 &1&x&y&x^2&xy&y^2&x^3&x^2y&xy^2&y^3
   &x^4&x^3y&x^2y^2&xy^3&y^4\\
f_1
 &0&0&0&1&0&0&0&0&0&0&0&0&0&0&0\\
f_2
 &0&0&0&0&0&0&0&0&0&1&0&0&0&0&0\\
xf_1
 &0&0&0&0&0&0&1&0&0&0&0&0&0&0&0\\
yf_1
 &0&0&0&0&0&0&0&1&0&0&0&0&0&0&0\\
xf_2
 &0&0&0&0&0&0&0&0&0&0&0&0&0&1&0\\
yf_2
 &0&0&0&0&0&0&0&0&0&0&0&0&0&0&1\\
x^2f_1
 &0&0&0&0&0&0&0&0&0&0&1&0&0&0&0\\
xyf_1
 &0&0&0&0&0&0&0&0&0&0&0&1&0&0&0\\
y^2f_1
 &0&0&0&0&0&0&0&0&0&0&0&0&1&0&0\\
};

\coordinate (NW) at (M-2-2.north west);
\draw (NW) rectangle (M-3-4.south east);
\draw (NW) rectangle (M-7-7.south east);
\draw (NW) rectangle (M-10-11.south east);
\draw (NW) rectangle (M-10-16.south east);
\end{tikzpicture}
\]
The table below applies the rank-nullity theorem to
$M_1,\ldots,M_4$.
\begin{center}
\begin{tabular}{r|cccc}
\toprule
 & $M_1$ & $M_2$ & $M_3$ & $M_4$\\
\midrule
\# Columns & $3$ & $6$ & $10$ & $15$\\
Rank              & $0$ & $1$ & $4$  & $9$\\
Nullity           & $3$ & $5$ & $6$  & $6$\\
\bottomrule
\end{tabular}
\end{center}
Thus $\mu = (3,5,6,6,\ldots)$ and 
$\textbf{0}$ is isolated with multiplicity $6$. The stabilization depth is $3$ and the detection depth is $4$.
Using \Cref{alg:quadratic-isolation}, we have $J=0$, and the quadratic obstruction equations are
$q=(w_1^2,0)$, vanishing along the $w_2$-axis. Thus, our quadratic test
is inconclusive.
\end{example}

\newpage

\begin{example}
Consider the following example 
from \cite[\S1.2]{BatesHauensteinPetersonSommese2009}:
\[
  f_1=x-y+x^2,\qquad f_2=x-y+y^2.
\]  At the origin,
$\ker(J)=\operatorname{span}\{(1,1)\}$ and
$F_2(w,w)=(w^2,w^2)\in\operatorname{im}(J)$, so the quadratic obstruction
vanishes identically.  The quadratic test is therefore inconclusive.

Below is $M_3$.  The inner boxes indicate $M_1$ and $M_2$.
\[
\begin{tikzpicture}[baseline=(M.center)]
\matrix (M) [
  matrix of math nodes,
  nodes={
    font=\footnotesize,
    minimum width=1.5em,
    minimum height=1em,
    inner sep=0pt
  },
  column sep=0pt,
  row sep=0pt
] {
 &1&x&y&x^2&xy&y^2&x^3&x^2y&xy^2&y^3\\
f_1
 &0&1&-1&1&0&0&0&0&0&0\\
f_2
 &0&1&-1&0&0&1&0&0&0&0\\
xf_1
 &0&0&0&1&-1&0&1&0&0&0\\
yf_1
 &0&0&0&0&1&-1&0&1&0&0\\
xf_2
 &0&0&0&1&-1&0&0&0&1&0\\
yf_2
 &0&0&0&0&1&-1&0&0&0&1\\
x^2f_1
 &0&0&0&0&0&0&1&-1&0&0\\
xyf_1
 &0&0&0&0&0&0&0&1&-1&0\\
y^2f_1
 &0&0&0&0&0&0&0&0&1&-1\\
x^2f_2
 &0&0&0&0&0&0&1&-1&0&0\\
xyf_2
 &0&0&0&0&0&0&0&1&-1&0\\
y^2f_2
 &0&0&0&0&0&0&0&0&1&-1\\
};

\coordinate (NW) at (M-2-2.north west);
\draw (NW) rectangle (M-3-4.south east);
\draw (NW) rectangle (M-7-7.south east);
\draw (NW) rectangle (M-13-11.south east);
\end{tikzpicture}
\]
The nullities are $(\mu_1,\mu_2,\mu_3)=(2,3,3)$, which proves that $\textbf{0}$ is isolated of multiplicity $3$.
\end{example}

Hao, Sommese, and Zeng improve this method by first constructing the
subspace of differential functionals satisfying certain closedness conditions and
then imposing the equations on that smaller space. This avoids constructing
the full multiplicity matrices and substantially reduces their combinatorial
growth. We do not describe their closedness-subspace algorithm here; instead,
we refer the reader to \cite{HaoSommeseZeng2013} for details.

\subsection{Additional examples and comparison with the quadratic obstruction test}
This section presents additional examples, including many from the existing literature, that showcase the difference between the classical local dimension test and \Cref{alg:quadratic-isolation}.

\begin{example}
\label{ex:four-dimensional-kernel}
For $N\geq 4$, let
\[
  F(\x)=
  (x_1,\ldots,x_{N-4},
   x_{N-3}^2,x_{N-2}^2,x_{N-1}^2,x_N^2)
\]
and take $\p=\zero$.  The Jacobian is
\[
  J=
  \begin{pmatrix}
    I_{N-4}&0\\
    0&0
  \end{pmatrix}.
\]
Thus $\ker(J)$ and $\ker(J^T)$ are both four-dimensional, spanned by the last four coordinate vectors.  Constructing $B$ and $L$ accordingly gives
\[
  q(\w)=L^TF_2(B\w)=(w_1^2,w_2^2,w_3^2,w_4^2).
\]
These equations have no common solution other than $\zero$, so $\p$ is isolated. 
\end{example}

\begin{example}[\texttt{CBMS2}]
\label{ex:cbms2}
Let $F=(f_1,f_2,f_3)$, where
\[
\begin{aligned}
 f_1&=(x-y)^3-z^2,\\
 f_2&=(z-x)^3-y^2,\\
 f_3&=(y-z)^3-x^2,
\end{aligned}
\qquad\text{and take }\p=\zero.
\]  This is the \texttt{CBMS2} benchmark of
\cite{HaoSommeseZeng2013}.
The Taylor expansion at $\p$ is
\[
  F(\p+\x)=F_1(\x)+F_2(\x)+F_3(\x) = 
  \underbrace{(0,0,0)}_{F_1(\x)}+
  \underbrace{(-z^2,-y^2,-x^2)}_{F_2(\x)}+
  \underbrace{((x-y)^3,(z-x)^3,(y-z)^3)}_{F_3(\x)}.
\]
Since $J=0$, we may take $B=L=I_3$.  Hence
\[
  L^TF_2(B\w)=(-w_3^2,-w_2^2,-w_1^2).
\]
The only common zero is $\zero$, so the quadratic test proves that $\p$ is
isolated.
\end{example}
\begin{example}[The KSS family]
\label{ex:kss}
For $n\geq2$, consider the original equations
\[
  g_i(\y)=y_i^2+\sum_{j=1}^n y_j-2y_i-(n-1),
  \qquad i=1,\ldots,n,
\]
at $\p^*=(1,\ldots,1)$.  Set $\y=\p^*+\x$ and define
\[
  f_i(\x)=g_i(\p^*+\x)=\sum_{j=1}^n x_j+x_i^2.
\]
We apply the test to $F=(f_1,\ldots,f_n)$ at $\p=\zero$.  Its linear and
quadratic parts are
\[
  F_1(\x)=\left(\sum_{j=1}^n x_j, \ldots, \sum_{j=1}^n x_j\right),
  \qquad
  F_2(\x)=(x_1^2,\ldots,x_n^2).
\]

The Jacobian $J$ is the $n\times n$ all-ones matrix. It has rank one, with
\[
  \ker(J)=\ker(J^T)
  =
  \left\{\x\in\CC^n\middle|\sum_{i=1}^n x_i=0\right\},
  \qquad
  \operatorname{im}(J)=\Span\{\mathbf{1}_n\}.
\]
We may take
\[
  B=L=
  \begin{pmatrix}
    1 & 0 & \cdots & 0 \\
    0 & 1 & \cdots & 0 \\
    \vdots &  & \ddots &  \\
    0 & 0 & \cdots & 1 \\
    -1 & -1&\cdots & -1 
  \end{pmatrix} \in \operatorname{Mat}_{n\times(n-1)}(\CC).
\]
The columns of both matrices form a basis of the corresponding kernel.  For
$\w=(w_1,\ldots,w_{n-1})$, write
  $s=w_1+\cdots+w_{n-1}$. 
Then
\[
  F_2(B\w)=(w_1^2,\ldots,w_{n-1}^2,s^2)
\]
and the quadratic obstruction equations $q=L^TF_2(B\w)$ are
\[
  q_i(\w)=w_i^2-s^2,
  \qquad i=1,\ldots,n-1.
\]
Suppose that $q_1(\w)=\cdots=q_{n-1}(\w)=0$.  If $s=0$, then
$w_i^2=0$ for every $i$, so $\w=\zero$.  If  $s\neq0$, then for each $i$ write
$  w_i=\epsilon_i s$ for $\epsilon_i\in\{1,-1\}$.
Since $s=\sum_{i=1}^{n-1}w_i$, division by $s$ gives
\[
  1=\sum_{i=1}^{n-1}\epsilon_i.
\]
If $n$ is odd, then this sum of $n-1$ signs cannot equal $1$, so necessarily $s=0$, in which case the quadratic test proves isolation.

If $n$ is even, then $n-1$ is odd, and one can choose
$\epsilon_1,\ldots,\epsilon_{n-1}$ whose sum is $1$.  For any $s\neq0$,
\[
  \w=(\epsilon_1s,\ldots,\epsilon_{n-1}s)
\]
is a nonzero solution of the quadratic obstruction equations.  Hence
the quadratic test is inconclusive. The cases $n=5,6$, and $7$ are among the KSS benchmarks considered in
\cite{HaoSommeseZeng2013}.  Our parity calculation shows that the quadratic
test certifies isolation for $n=5$ and $n=7$, but is inconclusive for $n=6$.
\end{example}

The next two examples revisit benchmark systems from
\cite{BatesHauensteinPetersonSommese2009}.

\begin{example}[Rhodonea intersections]
\label{ex:rhodonea}
For odd $k$, define
\[
  h_k(x,y)
  =
  (x^2+y^2)^{(k+1)/2}
  -
  \frac{(x+\mathrm{i}y)^k-(x-\mathrm{i}y)^k}{2\mathrm{i}}.
\]
Over $\mathbb R$, the equation $h_k=0$ defines the \mydef{rhodonea curve}, which has polar coordinate parametrization
$r=\sin(k\theta)$.  Let $R$ be a generic rotation about the origin and, for
odd $m,n\geq 3$, set
\[
  f_1=h_m,\qquad f_2=h_n\circ R^{-1},\qquad F=(f_1,f_2),
\]
with $\p=\zero$.  This defines the two randomly rotated rhodonea curves used
in \cite{BatesHauensteinPetersonSommese2009}.  The lowest-degree terms of
$f_1$ and $f_2$ have degrees $m,n \geq 3$, respectively. Hence
$F_1$ and $F_2$ both vanish identically. Thus, the quadratic test is inconclusive.  The classical test detects the
higher-order terms and computes the multiplicity $mn$.
\end{example}

\begin{example}[Two circles with a common linear factor]
\label{ex:circle-factor}
Let $g=x^2+y^2-y$
and let $\widehat g=g\circ R^{-1}$ for a generic rotation $R$ about the
origin.  Define
\[
  f_1=xg,\qquad f_2=x\widehat g,\qquad F=(f_1,f_2),
\]
and take $\p=\zero$.  If $-\ell$ is the linear part of $\widehat g$, then $F_2=(-xy,-x\ell)$.
The Jacobian at $\p$ is zero, so we take $B=L=I_2$.  Up to signs, the quadratic
obstructions are
\[
  q(\w)=
  \bigl(w_1w_2,\;w_1\ell(w_1,w_2)\bigr).
\]
They vanish along the line $w_1=0$, and so the quadratic test is inconclusive.
\end{example}

\subsection{Comparisons} 
We conclude with two tables comparing the quadratic test with classical isolation tests. \Cref{tab:conclusiveness} records the behavior of the quadratic test on the examples above. As expected, the test is inconclusive on the nonisolated example; it is also inconclusive on some isolated examples because their quadratic obstruction equations have nonzero solutions.

For each isolated example, \Cref{tab:comparison} reports the detection depth and the dimensions of the matrices at that depth involved in each algorithm.  For the equation-by-equation closedness-subspace algorithm of Hao, Sommese, and Zeng \cite{HaoSommeseZeng2013}, the largest matrix whose nullspace is taken is reported. For the classical method of Bates, Hauenstein, Peterson, and Sommese~\cite{BatesHauensteinPetersonSommese2009}, we report the size of the last multiplicity matrix (with zero rows removed).

\begin{table}[H]
\centering
\caption{Behavior of the quadratic test on the examples in this subsection.}
\label{tab:conclusiveness}
\small
\setlength{\tabcolsep}{3pt}
\begin{tabular}{@{}l|lll@{}}
\toprule
System & Example  & Output & Reason\\
\midrule
four-dimensional kernel
  & \cref{ex:four-dimensional-kernel}
  & isolated & $w_1^2=\cdots=w_4^2=0$\\
\texttt{CBMS2} \cite{HaoSommeseZeng2013}
  & \cref{ex:cbms2}
  & isolated &  $-w_1^2=\cdots=-w_3^2=0$\\
KSS, odd $n$ \cite{HaoSommeseZeng2013}
  & \cref{ex:kss}
  & isolated & no balanced sign vector\\ \hline
KSS, even $n$ \cite{HaoSommeseZeng2013}
  & \cref{ex:kss}
  & inconclusive & balanced sign vectors\\
rhodonea curves \cite{BatesHauensteinPetersonSommese2009}
  & \cref{ex:rhodonea}
  & inconclusive & $F_2$ is identically zero\\ \hline
two circles \cite{BatesHauensteinPetersonSommese2009}
  & \cref{ex:circle-factor}
  & inconclusive & not isolated\\
\bottomrule
\end{tabular}
\end{table}

\begin{table}[H]
\centering
\caption{Matrix dimensions  at the detection depth for two classical isolation tests. For the Hao--Sommese--Zeng method, we report the largest nullspace matrix encountered at that depth. For the multiplicity-matrix method, zero rows are removed. Bold indicates \Cref{alg:quadratic-isolation} succeeds.}
\label{tab:comparison}
\small
\setlength{\tabcolsep}{5pt}
\begin{tabular}{@{}l|lll@{}}
\toprule
System & Detection depth & Hao--Sommese--Zeng & Classical\\
\midrule
$\langle x^2,y^3\rangle$
  & $4$ & $4\times 6$ & $9\times 15$\\ \hline
{BHPS example}
  & $3$ & $3\times 6$ & $12\times 10$\\ \hline
\textbf{four-dimensional kernel}, $N=4$
  & $5$ & $7\times 26$ & $140\times 126$\\
\textbf{four-dimensional kernel}, $N=5$
  & $5$ & $20\times 16$ & $350\times 252$\\
\textbf{four-dimensional kernel}, $N=6$
  & $5$ & $21\times 16$ & $756\times 462$\\ \hline
\textbf{\texttt{CBMS2}}
  & $4$ & $6\times 8$ & $30\times 35$\\ \hline
KSS, $n=2$
  & $3$ & $3\times 6$ & $12\times 10$\\
\textbf{KSS}, $n=3$
  & $3$ & $7\times 9$ & $30\times 20$\\
KSS, $n=4$  & $5$ & $65\times 24$ & $280\times 126$\\
\textbf{KSS}, $n=5$  & $5$ & $121\times 35$ & $630\times 252$\\
KSS, $n=6$  & $7$ & $917\times 88$ & $5544\times 1716$\\
\textbf{KSS}, $n=7$  & $7$ & $1688\times 136$ & $12012\times 3432$\\ \hline
rhodonea, $(5,5)$
  & $9$ & $28\times 48$ & $30\times 55$\\
rhodonea, $(7,7)$
  & $13$ & $66\times 96$ & $56\times 105$\\
rhodonea, $(9,9)$
  & $17$ & $120\times 160$ & $90\times 171$\\
\bottomrule
\end{tabular}
\end{table}

For our capstone example in the next section of the $24$-cell,  the multiplicity matrix $M_2$ has size $37,056 \times 18,721$. The reduction method of \cite{HaoSommeseZeng2013} involves much smaller matrices at the cost of performing many more nullspace computations. Ultimately it does not finish in a reasonable amount of time.
\section{The regular 24-cell}\label{sec:24cell}

A realization of a combinatorial $d$-polytope $P$ is a vertex and facet labeled
convex polytope in $\mathbb R^d$ whose face lattice agrees with the face lattice of
$P$.  A realization space consists of all
such coordinate choices in a specified model. By the realization space of $P$, we mean the algebraic closure of this set over $\mathbb{C}$.

Following \cite{RastanawiSinnZiegler2021}, we use the centered
vertex-facet model.  A realization is \mydef{centered} when the origin lies
in its interior.  Its facet inequalities can then be normalized as
$a_j(\x)\leq1$.  Thus a centered realization consists of vertex
vectors $\v_i\in\mathbb R^4$ and facet covectors
$a_j\in(\mathbb R^4)^*$ satisfying
\[
  a_j(\v_i)=1
\]
at every vertex-facet incidence and $a_j(\v_i)<1$ at every nonincidence. We write $\mathcal X_{P}^\circ$ for the set of all centered realizations in $\mathbb{R}^N$, where $N=d(|\textrm{vertices}(P)|+|\textrm{facets}(P)|)$, $\mathcal X_P$ for its algebraic closure in $\mathbb{C}^N$, and $\mathcal X_P(\mathbb{R})$ for $\mathcal X_P \cap \mathbb{R}^N$.

We now restrict our attention to the \mydef{$24$-cell}. It is one of the six regular convex $4$-polytopes.  It has $24$
vertices, $96$ edges, $96$ triangular $2$-faces, and $24$ octahedral facets.
Write $P$ for its labeled combinatorial type. It has $144$
vertex-facet incidences.  The vertex-facet incidence matrix is given as follows:
\[
{\mathcal I = \tiny{\begin{pmatrix}
1&1&1&0&0&0&0&0&0&0&0&0&0&1&0&0&0&0&0&1&1&0&0&0\\
1&0&1&0&1&0&0&1&0&0&0&0&0&0&0&0&0&0&1&1&0&0&0&0\\
0&0&1&0&0&0&0&0&0&0&0&0&1&0&0&1&0&0&1&1&1&0&0&0\\
1&0&0&0&0&0&0&1&0&1&0&0&0&1&0&0&1&0&0&1&0&0&0&0\\
0&0&0&0&0&0&0&0&0&1&0&0&1&1&1&0&0&0&0&1&1&0&0&0\\
1&1&1&1&1&1&0&0&0&0&0&0&0&0&0&0&0&0&0&0&0&0&0&0\\
0&0&0&0&0&0&0&1&0&1&0&0&1&0&0&0&0&0&1&1&0&0&0&1\\
0&1&1&1&0&0&1&0&0&0&0&0&0&0&0&1&0&0&0&0&1&0&0&0\\
0&0&1&1&1&0&0&0&0&0&0&0&0&0&0&1&0&1&1&0&0&0&0&0\\
1&1&0&0&0&1&0&0&0&0&0&0&0&1&0&0&1&0&0&0&0&0&1&0\\
1&0&0&0&1&1&0&1&0&0&0&0&0&0&0&0&1&0&0&0&0&1&0&0\\
0&1&0&0&0&0&1&0&0&0&0&0&0&1&1&0&0&0&0&0&1&0&1&0\\
0&0&0&0&0&0&1&0&1&0&0&0&1&0&1&1&0&0&0&0&1&0&0&0\\
0&0&0&0&1&0&0&1&0&0&0&0&0&0&0&0&0&1&1&0&0&1&0&1\\
0&0&0&0&0&0&0&0&1&0&0&0&1&0&0&1&0&1&1&0&0&0&0&1\\
0&0&0&0&0&0&0&0&0&1&0&1&0&1&1&0&1&0&0&0&0&0&1&0\\
0&0&0&0&0&0&0&1&0&1&0&1&0&0&0&0&1&0&0&0&0&1&0&1\\
0&1&0&1&0&1&1&0&0&0&1&0&0&0&0&0&0&0&0&0&0&0&1&0\\
0&0&0&0&0&0&0&0&1&1&0&1&1&0&1&0&0&0&0&0&0&0&0&1\\
0&0&0&1&1&1&0&0&0&0&1&0&0&0&0&0&0&1&0&0&0&1&0&0\\
0&0&0&1&0&0&1&0&1&0&1&0&0&0&0&1&0&1&0&0&0&0&0&0\\
0&0&0&0&0&1&0&0&0&0&1&1&0&0&0&0&1&0&0&0&0&1&1&0\\
0&0&0&0&0&0&1&0&1&0&1&1&0&0&1&0&0&0&0&0&0&0&1&0\\
0&0&0&0&0&0&0&0&1&0&1&1&0&0&0&0&0&1&0&0&0&1&0&1
\end{pmatrix}
}}
\] 
Let $\mathcal X\subseteq\CC^{4(24+24)} = \CC^{192}$ be the \mydef{incidence variety} defined by the
$144$ incidence equations:
\begin{equation}
\label{eq:24cellequations}
  g_{ij}(\mathbf a,\v)=\mathbf a_j^T\v_i-1,
  \qquad (i,j)\in\{(i,j) \mid \mathcal I_{i,j}=1\}.
\end{equation}  Its expected dimension is $192-144=48$. Since every centered realization satisfies
\eqref{eq:24cellequations}, $
  \mathcal X_P\subseteq\mathcal X
$.

The \mydef{symmetric regular realization} of the $24$-cell is
\[
  P_{\mathrm{reg}}
  =
  \operatorname{conv}
  \{\pm\mathbf e_i\pm\mathbf e_j\mid 1\leq i<j\leq4\}.
\]
Let $\p^*=(\mathbf a_1^*,\ldots,\mathbf a_{24}^*,
            \v_1^*,\ldots,\v_{24}^*)  \in\mathcal X_P$ represent
$P_{\mathrm{reg}}$.    At
$\p^*$, the Jacobian has rank $140$, so the Zariski tangent space has
dimension $52$.  Rastanawi, Sinn, and Ziegler constructed smooth real points
of local dimension $48$ converging to $\p^*$ and proved that
$\p^*$ is singular
\cite{RastanawiSinnZiegler2021}.  Consequently,
\[
  48
  \leq\dim_{\p^*}(\mathcal X_P)
  \leq\dim_{\p^*}(\mathcal X)
  \leq52.
\]

We intersect $\mathcal X$ with a codimension-$48$ coordinate linear space through
$\p^*$.  Namely, we fix the following twelve vertices at their coordinates
in the symmetric regular realization:
\[
\begin{aligned}
 \v_1^*&=(-1,-1,0,0),&
 \v_2^*&=(-1,0,-1,0),&
 \v_3^*&=(-1,0,0,-1),\\
 \v_7^*&=(-1,1,0,0),&
 \v_{18}^*&=(1,-1,0,0),&
 \\[3pt]
 \v_6^*&=(0,-1,-1,0),&
 \v_{10}^*&=(0,-1,0,1),&
 \v_{15}^*&=(0,1,0,-1),\\
 \v_{19}^*&=(0,1,1,0),&
 \v_{22}^*&=(1,0,0,1),&
 \v_{23}^*&=(1,0,1,0),\\
 \v_{24}^*&=(1,1,0,0).&&
\end{aligned}
\]
Thus the slice $\mathcal L$ is defined by the $48$ scalar equations
\[
  \v_i-\v_i^*=\zero,
  \qquad
  i\in\{1,2,3,6,7,10,15,18,19,22,23,24\}.
\]
The first five fixed vertices form an affine frame. We now write
\[
  G:\CC^{192}\longrightarrow\CC^{192}
\]
for the polynomial map consisting of the $144$ incidence equations followed
by these $48$ slice equations. We translate $\p^*$ to the origin by setting
\[
  F(\x)=G(\p^*+\x),
  \qquad \text{ and }  \qquad \p=\zero.
\]
The exact computation in \texttt{24cell.m2} shows that $J$ has rank $188$, and so $b=c=4$. We construct $B$ and $L$ using the following \texttt{Macaulay2} commands:
  \[
  \texttt{B = gens ker J} \quad \quad \text{ and } \quad \quad \texttt{L = gens ker transpose J}.
  \]
Our code computes the following quadratic obstruction equations in the variables $\w=(w_1,\ldots,w_4)$ parametrizing the kernel of $J$:
\begin{align*}
q_1 &=
 -4w_1w_2-4w_1w_3+4w_2w_3+8w_3^2-16w_3w_4,\\
q_2 &=
 8w_2^2+8w_2w_3+8w_3^2-12w_2w_4-12w_3w_4,\\
q_3 &=
 -4w_1w_2-4w_1w_3+4w_2w_3+8w_3^2
 +8w_1w_4-8w_2w_4-8w_3w_4+4w_4^2,\\
q_4 &=
 8w_1^2-8w_1w_3+8w_3^2+12w_1w_4-12w_3w_4.
\end{align*}
The exact computation returns the dimension of this homogeneous ideal:
\[
  \texttt{dim(Iq)}=0.
\]
The quadratic test therefore proves that
$\p^*$ is isolated in $\mathcal X\cap\mathcal L$.
\begin{proof}[Proof of \Cref{thm:24cell}]
The incidence variety $\mathcal X\subseteq\CC^{192}$ is cut out by $144$
equations.  Therefore,
\[
  \dim_{\p^*}(\mathcal X)\geq192-144=48.
\]
On the other hand, $\mathcal L$ is defined by $48$ linearly independent
coordinate equations through $\p^*$.  The exact quadratic computation above
proves that $\p^*$ is isolated in $\mathcal X\cap\mathcal L$.  The slicing
inequality therefore gives $
  \dim_{\p^*}(\mathcal X)\leq 48$ and hence $
  \dim_{\p^*}(\mathcal X)=48$.

Since $\mathcal X_P\subseteq\mathcal X$, it follows that
$
  \dim_{\p^*}(\mathcal X_P)\leq48$. Rastanawi, Sinn, and Ziegler constructed smooth real realizations converging
to $\p^*$ along a $48$-dimensional component of the realization space \cite{RastanawiSinnZiegler2021}.  Therefore,
\[
  \dim_{\p^*}(\mathcal X_P)\geq48,
\]
and consequently $
  \dim_{\p^*}(\mathcal X_P)=48.$
\end{proof}

\newpage

\begingroup
\footnotesize
\bibliographystyle{abbrv}
\bibliography{references}
\bigskip

\noindent
\textsc{Taylor Brysiewicz}\\
Department of Mathematics, Western University\\
London, Ontario N6A 5B7, Canada\\
\textit{Email address:} \texttt{tbrysiew@uwo.ca}

\medskip

\noindent
\textsc{Rainer Sinn}\\
Mathematisches Institut, Georg-August-Universit\"at G\"ottingen\\
Bunsenstra{\ss}e 3--5, 37073 G\"ottingen, Germany\\
\textit{Email address:}
\texttt{rainer.sinn@mathematik.uni-goettingen.de}
\endgroup

\appendix

\newpage

\section{Exact computation for the symmetric regular $24$-cell}
\label{app:24cell-computation}

This appendix contains the exact \texttt{Macaulay2} computation used in the
proof of \Cref{thm:24cell}.  All calculations are performed over $\QQ$.
The file \texttt{24cell-data.m2} contains the vertex-facet incidence matrix,
the ordered vertex and facet coordinates of the symmetric regular
realization, and the indices of the twelve fixed vertices.  The file
\texttt{24cell.m2} constructs the $144$ incidence equations and the $48$
coordinate slice equations, evaluates their Jacobian at the symmetric regular
realization, and constructs the quadratic obstruction ideal.
The two files should be placed in the same directory.  The computation
returns
\[ 
  \dim(I_q)=0.
\]
\subsection{Supplementary Code}

\begin{lstlisting}[
  language=M2,
  basicstyle=\footnotesize\ttfamily,
  breaklines=true,
  caption={The file \texttt{24cell-data.m2}.}
]
inc = matrix(ZZ,{
 {1,1,1,0,0,0,0,0,0,0,0,0,0,1,0,0,0,0,0,1,1,0,0,0}, 
 {1,0,1,0,1,0,0,1,0,0,0,0,0,0,0,0,0,0,1,1,0,0,0,0},
 {0,0,1,0,0,0,0,0,0,0,0,0,1,0,0,1,0,0,1,1,1,0,0,0},
 {1,0,0,0,0,0,0,1,0,1,0,0,0,1,0,0,1,0,0,1,0,0,0,0},
 {0,0,0,0,0,0,0,0,0,1,0,0,1,1,1,0,0,0,0,1,1,0,0,0},
 {1,1,1,1,1,1,0,0,0,0,0,0,0,0,0,0,0,0,0,0,0,0,0,0},
 {0,0,0,0,0,0,0,1,0,1,0,0,1,0,0,0,0,0,1,1,0,0,0,1},
 {0,1,1,1,0,0,1,0,0,0,0,0,0,0,0,1,0,0,0,0,1,0,0,0},
 {0,0,1,1,1,0,0,0,0,0,0,0,0,0,0,1,0,1,1,0,0,0,0,0},
 {1,1,0,0,0,1,0,0,0,0,0,0,0,1,0,0,1,0,0,0,0,0,1,0},
 {1,0,0,0,1,1,0,1,0,0,0,0,0,0,0,0,1,0,0,0,0,1,0,0},
 {0,1,0,0,0,0,1,0,0,0,0,0,0,1,1,0,0,0,0,0,1,0,1,0},
 {0,0,0,0,0,0,1,0,1,0,0,0,1,0,1,1,0,0,0,0,1,0,0,0},
 {0,0,0,0,1,0,0,1,0,0,0,0,0,0,0,0,0,1,1,0,0,1,0,1},
 {0,0,0,0,0,0,0,0,1,0,0,0,1,0,0,1,0,1,1,0,0,0,0,1},
 {0,0,0,0,0,0,0,0,0,1,0,1,0,1,1,0,1,0,0,0,0,0,1,0},
 {0,0,0,0,0,0,0,1,0,1,0,1,0,0,0,0,1,0,0,0,0,1,0,1},
 {0,1,0,1,0,1,1,0,0,0,1,0,0,0,0,0,0,0,0,0,0,0,1,0},
 {0,0,0,0,0,0,0,0,1,1,0,1,1,0,1,0,0,0,0,0,0,0,0,1},
 {0,0,0,1,1,1,0,0,0,0,1,0,0,0,0,0,0,1,0,0,0,1,0,0},
 {0,0,0,1,0,0,1,0,1,0,1,0,0,0,0,1,0,1,0,0,0,0,0,0},
 {0,0,0,0,0,1,0,0,0,0,1,1,0,0,0,0,1,0,0,0,0,1,1,0},
 {0,0,0,0,0,0,1,0,1,0,1,1,0,0,1,0,0,0,0,0,0,0,1,0},
 {0,0,0,0,0,0,0,0,1,0,1,1,0,0,0,0,0,1,0,0,0,1,0,1}
 })
h = 1/2
regularVertices={{-1,-1,0,0},{-1,0,-1,0},{-1,0,0,-1},{-1,0,0,1},{-1,0,1,0},
{0,-1,-1,0},{-1,1,0,0},{0,-1,0,-1},{0,0,-1,-1},{0,-1,0,1},{0,0,-1,1},
{0,-1,1,0},{0,0,1,-1},{0,1,-1,0},{0,1,0,-1},{0,0,1,1},{0,1,0,1},
{1,-1,0,0},{0,1,1,0},{1,0,-1,0},{1,0,0,-1},{1,0,0,1},{1,0,1,0},
{1,1,0,0}}
regularFacets={{-h,-h,-h,h},{0,-1,0,0},{-h,-h,-h,-h},{h,-h,-h,-h},{0,0,-1,0},
{h,-h,-h,h},{h,-h,h,-h},{-h,h,-h,h},{h,h,h,-h},{-h,h,h,h},{1,0,0,0},
{h,h,h,h},{-h,h,h,-h},{-h,-h,h,h},{0,0,1,0},{0,0,0,-1},{0,0,0,1},
{h,h,-h,-h},{-h,h,-h,-h},{-1,0,0,0},{-h,-h,h,-h},{h,h,-h,h},{h,-h,h,h},
{0,1,0,0}}
frameVertices = {1,2,3,7,18}
additionalFixedVertices = {6,10,15,19,22,23,24}
fixedVertices = join(frameVertices,additionalFixedVertices)
\end{lstlisting}

\begin{lstlisting}[
  language=M2,
  basicstyle=\footnotesize\ttfamily,
  breaklines=true,
  caption={The file \texttt{24cell.m2}.}
]

--Load 24-cell incidence and realization of interest
load "24cell-data.m2"
--create realization space ring
R = QQ[a_(1,1)..a_(24,4),v_(1,1)..v_(24,4)]
ringVariables = flatten entries vars R
--helpers
aIndex = (j,k) -> 4*(j-1)+(k-1); 
vIndex = (i,k) -> 96+4*(i-1)+(k-1);
aVar = (j,k) -> ringVariables#(aIndex(j,k)); 
vVar = (i,k) -> ringVariables#(vIndex(i,k));
--standard regular realization 
pstar = join(flatten regularFacets, flatten regularVertices);
--vertex/facet incidence equation constructor for realization space
incidenceEquation = ij -> (
    i := ij#0;
    j := ij#1;
    sum(toList(1..4), k -> aVar(j,k)*vVar(i,k)) - 1
);
incidencePairs = {}
for i from 1 to 24 do (for j from 1 to 24 do (
        if inc_(i-1,j-1) == 1 then
            incidencePairs = append(incidencePairs,{i,j});
););
--equations
incidenceEquations = apply(incidencePairs,incidenceEquation)
--additional equations coming from a coordinate linear space
sliceEquations = flatten apply(fixedVertices, i ->
    apply(toList(1..4), k ->
        vVar(i,k)-(regularVertices#(i-1))#(k-1)
))
--construct Jacobian and evaluate at pstar
F = join(incidenceEquations,sliceEquations)
J = transpose jacobian matrix{F}
atP = map(QQ,R,pstar)
A = atP J
--Construct B and L
B = gens ker A; L = gens ker transpose A
--create quadratic obstruction equations
W = QQ[w_1..w_(numColumns(B))]
wVariables = flatten entries vars W
wColumn = transpose matrix{wVariables}
Bw = sub(B,W)*wColumn
deltaA = (j,k) -> Bw_(aIndex(j,k),0); deltaV = (i,k) -> Bw_(vIndex(i,k),0);
quadraticIncidenceValues = apply(incidencePairs, ij -> (
        i := ij#0; j := ij#1;
        sum(toList(1..4), k ->
            deltaA(j,k)*deltaV(i,k)
)))
quadraticValues = join(quadraticIncidenceValues,48:0_W)
F2Bw = transpose matrix{quadraticValues}
qColumn = transpose(sub(L,W))*F2Bw

--turn equations into an ideal and verify 0-dimensional
Iq = ideal flatten entries qColumn
dim(Iq) --outputs 0
\end{lstlisting}

\end{document}